\documentclass[final,leqno]{siamltex704}
\usepackage{amsmath}
\usepackage{amssymb}
\usepackage{bm}
\usepackage{color}
\usepackage{graphicx}
\usepackage{graphicx,subfigure}
\usepackage{mathtools}
\usepackage{hyperref}
\usepackage[notcite,notref]{showkeys}
\usepackage{mathrsfs}
\usepackage{tikz}
\newcommand{\bv}{{\bm v}}
\newcommand{\bu}{{\bm u}}
\newcommand{\bn}{{\bm n}}

\newcommand{\bw}{{\bm w}}
\newcommand{\bt}{{\bm t}}

\newcommand{\pT}{{\partial T}}

\def\l{{\langle}}
\def\r{{\rangle}}
\def\3bar{{|\hspace{-.02in}|\hspace{-.02in}|}}

\renewcommand{\ldots}{\dotsc}
\newcommand{\kernel}{\operatorname{Ker}}
\newcommand{\range}{\operatorname{Range}}

\usepackage{float}
\usepackage{mathtools}

\title{De Rham Complexes for Weak Galerkin Finite Element Spaces}

\author{Chunmei Wang\thanks{Department of Mathematics \& Statistics, Texas Tech University, Lubbock, TX 79409, USA (chunmei.wang@ttu.edu).
The research of Chunmei Wang was partially supported by National Science Foundation Award DMS-1849483.}
\and Junping Wang\thanks{Division of Mathematical Sciences, National Science Foundation, Alexandria, VA 22314
(jwang@nsf.gov). The research of Junping Wang was supported by the NSF IR/D program, while working at National Science Foundation. However, any opinion, finding, and conclusions or recommendations expressed in this material are those of the author and do not necessarily reflect the views of the National Science Foundation.}\and Xiu Ye\thanks{Department of
Mathematics, University of Arkansas at Little Rock, Little Rock, AR
72204 (xxye@ualr.edu). This research was supported in part by
National Science Foundation Grant DMS-1620016.}
\and
Shangyou Zhang\thanks{Department of
Mathematical Sciences, University of Delaware, Newark, DE 19716 (szhang@udel.edu).}
}

\begin{document}

\maketitle

\begin{abstract}
Two de Rham complex sequences of the finite element spaces are introduced for weak finite element functions and weak derivatives developed in the weak Galerkin (WG) finite element methods on general polyhedral elements. One of the sequences uses polynomials of equal order for all the finite element spaces involved in the sequence and the other one uses polynomials of naturally decending orders. It is shown that the diagrams in both de Rham complexes commute for general polyhedral elements. The exactness of one of the complexes is established for the lowest order element.
\end{abstract}



\section{Introduction}

Finite element exterior calculus is a framework that explores the structural properties of finite element aproximating spaces and their connection with basic differential operators such as gradient, curl, and divergence operators in differential calculus. Such differential operators usually form the basis of mathematical models for physical, engineering, and biological problems.  The classical conforming finite element exterior calculus has provided good guidance in the design and analysis
of various finite element methods for solving partial differential equations (PDE). In particular, the framework has proved to be powerful in analysing well-posedness of finite element discretizations and revealing new finite elements for solving various PDE modeling problems arising from science and engineering. Finite element exterior calculus has been well developed for conforming finite element methods on simplicial elements, and they often appear in the form of de Rham complexes in the application of numerical solutions for PDEs \cite{ay2015,
afw,afw1,boffi2001,boss1998,chh2018,dmvr2000,demk2008,pf2013,hipt1999,ghz2019,neilan2015,tw2006,whit1957}. 

In the last two decades, an extensive research effort has been made by the computational mathematics community in the development of finite element methods with discontinuous approximating functions. The weak Galerkin (WG) finite element method, first introduced in \cite{wy}, is one of the recently developed finite element techniques based on discontinuous finite element functions. The essence of weak Galerkin finite element methods is the use of weak finite element
functions and their weak derivatives defined as discrete distributions in polynomial subspaces. In general, weak Galerkin finite element formulations for partial differential equations can be derived naturally by replacing usual derivatives by discrete weak derivatives in the corresponding variational forms, with the option of adding stabilization term(s) to enforce a weak continuity of the approximating functions. Like most discontinuous finite element methods, WG method is applicable for finite element partitions with arbitrary shape of polygons/polyhedra.

The purpose of this paper is to present two de Rham diagrams for weak Galerkin finite element spaces with weakly defined differential operators. In particular, two finite element differential complexes are developed in the WG context on polyhedral elements. It is proved that the diagrams in the two de Rham complexes commute with simple, yet powerful $L^2$ projection operators from the continuous functional spaces to the corresponding WG finite element ``subspaces".

 {\sc WG de Rham Complex 1: \ Polynomials of Equal Order $k\ge 0$}

\begin{figure}[htb]
\begin{center}
\begin{tikzpicture}[scale=1.4]
 \node (AA) at (-1,1) {${\mathbb R}^1$};
 \node (A) at (0,1) {$C^\infty(T)$};
 \node (B) at (1.7,1) {$[{C}^\infty(T)]^3$};
 \node (C) at (3.4,1) {$[{C}^\infty(T)]^3$};
 \node (D) at (5.1,1) {$C^\infty(T)$};
 \node (I) at (6.0,1) {$0$};
 \node (EE) at (-1,0)   {${\mathbb R}^m$};
 \node (E) at (0,0)   {$V_{k}^{(1)}(T)$};
 \node (F) at (1.7,0) {$V_{k}^{(2)}(T)$};
 \node (G) at (3.4,0) {$V_{k}^{(3)}(T)$};
 \node (H) at (5.1,0) {$V_{k}^{(4)}(T)$};
 \node (II) at (6.0,0) {$0$};
 \draw[->] (AA) to node [above]{$I$}(A);
 \draw[->] (A) to node [above]{$\nabla$}(B);
 \draw[->] (B) to node [above]{$\nabla\times$}(C);
 \draw[->] (C) to node [above]{$\nabla\cdot$}(D);
 \draw[->] (D) to node [above]{$N$}(I);
 \draw[->] (EE) to node [above]{$I_w$}(E);
 \draw[->] (E) to node [above]{$\nabla_w$}(F);
 \draw[->] (F) to node [above]{$\nabla_w\times$}(G);
 \draw[->] (G) to node [above]{$\nabla_w\cdot$}(H);
 \draw[->] (A) to node [left]{$Q_h^{(1)}$}(E);
 \draw[->] (B) to node [left]{$Q_h^{(2)}$}(F);
 \draw[->] (C) to node [left]{$Q_h^{(3)}$}(G);
 \draw[->] (D) to node [left]{$Q_h^{(4)}$}(H);
 \draw[->] (H) to node [above]{$N$}(II);
 \end{tikzpicture}
\end{center}
\caption{ \label{d1} The WG de Rham Complex 1. }
\end{figure}

In the de Rham complex 1 (see Fig. \ref{d1}), the degree of polynomials in all four WG finite element spaces $V_k^{(i)}$, $ i=1,\ldots,4$, are of the same value $k\ge 0$, and $C^\infty(T)$ is the space of infinitely smooth functions in the closed region $T$. This special feature of using polynomials of equal order in all the finite element spaces is possible only for weakly defined differential operators, as the weak derivative of a $k$th order polynomial can be a polynomials of any degree in the WG context. Polynomials in the finite element spaces in the WG de Rham Complex 2 (see Fig. \ref{d2}) are in a naturally descending order. Details about these finite elmenet spaces and the weak differential operators can be found in forthcoming sections.

{\sc WG de Rham Complex 2: \ Polynomials of Descending Order $k\ge 3$}

\begin{figure}[htb]
\begin{center}
\begin{tikzpicture}[scale=1.4]
 \node (AA) at (-1,1) {${\mathbb R}^1$};
 \node (A) at (0,1) {$C^\infty(T)$};
 \node (B) at (1.7,1) {$[{C}^\infty(T)]^3$};
 \node (C) at (3.4,1) {$[{C}^\infty(T)]^3$};
 \node (D) at (5.1,1) {$C^\infty(T)$};
 \node (I) at (6.0,1) {$0$};
 \node (EE) at (-1,0)   {${\mathbb R}^n$};
 \node (E) at (0,0)   {$W_{k}^{(1)}(T)$};
 \node (F) at (1.7,0) {$W_{k-1}^{(2)}(T)$};
 \node (G) at (3.4,0) {$W_{k-2}^{(3)}(T)$};
 \node (H) at (5.1,0) {$W_{k-3}^{(4)}(T)$};
 \node (II) at (6.0,0) {$0$};
 \draw[->] (AA) to node [above]{$I$}(A);
 \draw[->] (A) to node [above]{$\nabla$}(B);
 \draw[->] (B) to node [above]{$\nabla\times$}(C);
 \draw[->] (C) to node [above]{$\nabla\cdot$}(D);
 \draw[->] (D) to node [above]{$N$}(I);
 \draw[->] (EE) to node [above]{$I_w$}(E);
 \draw[->] (E) to node [above]{$\nabla_w$}(F);
 \draw[->] (F) to node [above]{$\nabla_w\times$}(G);
 \draw[->] (G) to node [above]{$\nabla_w\cdot$}(H);
 \draw[->] (A) to node [left]{$R_h^{(1)}$}(E);
 \draw[->] (B) to node [left]{$R_h^{(2)}$}(F);
 \draw[->] (C) to node [left]{$R_h^{(3)}$}(G);
 \draw[->] (D) to node [left]{$R_h^{(4)}$}(H);
 \draw[->] (H) to node [above]{$N$}(II);
 \end{tikzpicture}
\end{center}
\caption{ \label{d2} The WG de Rham Complex 2. }
\end{figure}

The above two WG de Rham complexes have some unique features: (1) they both work for discontinuous approximation on general polyhedral elements; (2) all the cochain projections $Q_h^{(i)}$ in Fig. \ref{d1} and $R_h^{(i)}$ in Fig. \ref{d2} are the standard $L^2$ projections; and (3) weakly defined differential operators are employed first time in finite element differential complexes. It should be noted that the weak gradient $\nabla_w$, weak curl $\nabla_w\times$ and weak divergence $\nabla_w\cdot$ in the WG de Rham complexes 1 and 2 have been employed for solving various partial differential equations in many existing literatures, including \cite{mwyz,ww,wymix}. The WG de Rham complexes 1 and 2 additionally provide discrete weak differential operators defined on surfaces, which is of great interest to the development of numerical methods for PDEs on surfaces or manifolds in general.

\section{WG de Rham Complexes}

Let $T$ be a shape-regular polyhedron in the sense of \cite{wymix}. Denote by $F(T)$, $E(T)$ and $V(T)$ the set of faces, edges, and vertices of $T$, respectively. For any face $f\in F(T)$, let $\bn$ be a unit normal vector to $f$ and $\bt$ be a unit tangential vector on $e\subset\partial f$ which obey the right-hand rule. Throughout this paper, we adopt the notation of $(\cdot,\cdot)_{D}$ for the $L^2$ inner product in $L^2(D)$, where $D$ could be a volume, a surface, or a curve.

\subsection{WG de Rham complex for polynomials of equal order}

For a given non-negative integer $k\ge 0$, we introduce four vector spaces $V_k^{(1)}(T)$, $V_k^{(2)}(T)$, $V_k^{(3)}(T)$ and $V_k^{(4)}(T)$ that allow the operator operations shown as in Fig. \ref{d1}.
The first space $V_k^{(1)}(T)$ is defined by
\begin{eqnarray}
V_k^{(1)}(T)&=&\{v=\{v_0,v_f,v_e, v_n\}: v_0\in P_k(T),\;v_f\in P_{k}(f),\; v_e\in P_{k}(e),\nonumber\\
&&v_n\in \mathbb{R},\;f\in F(T), \; e\in E(T), n\in V(T)\}.\label{V1}
\end{eqnarray}
The second space $V_k^{(2)}(T)$ is given by
\begin{eqnarray}
V_k^{(2)}(T)&=&\{ \bu=\{\bu_0, \bu_f,\bu_e\}:\;\bu_0\in V_{k,0}^{(2)}(T), \bu_f\in V_{k,f}^{(2)}(f),\nonumber\\
&& \bu_e\in V_{k,e}^{(2)}(e),\;f\in F(T),\; e\in E(T)\},\label{V2}
\end{eqnarray}
where
\begin{eqnarray}
V_{k,0}^{(2)}(T)&=&[P_k(T)]^3,\label{V20}\\
V_{k,f}^{(2)}(f)&=&\{\bu_f=u_1\bt_{f,1}+u_2\bt_{f,2}:\; u_1,u_2\in P_k(f)\}, \label{V2f}\\
V_{k,e}^{(2)}(e)&=&\{\bu_e=u\bt_e: u\in P_k(e)\},\label{V2e}
\end{eqnarray}
$\bt_{f,1}$ and $\bt_{f,2}$ in (\ref{V2f}) are two orthogonal unit vectors  tangential to $f$, $\bt_e$ in (\ref{V2e}) is a tangent vector on $e$.
The third space $V_k^{(3)}(T)$ is defined as
\begin{eqnarray}
V_k^{(3)}(T)&=&\{ \bw=\{\bw_0, \bw_f\}: \bw_0 \in V_{k,0}^{(3)}(T),\; \bw_f\in V_{k,f}^{(3)}(f),\ f\in F(T)\},\label{V3}
\end{eqnarray}
where
\begin{eqnarray}
V_{k,0}^{(3)}(T)&=&[P_k(T)]^3,\label{V30}\\
V_{k,f}^{(3)}(f)&=&\{\bw_f=w_f\bn_f:\;  w_f\in P_k(f)\}, \label{V3f}
\end{eqnarray}
with $\bn_f$  being a unit normal vector to the surface $f$. The three unit vectors $\bt_{f,1}$, $\bt_{f,2}$, and $\bn_f$ are assumed to form an orthogonal right-hand system, though this assumption is not necessary.
Our fourth space $V_k^{(4)}(T)$ is given by
\begin{eqnarray}
V_k^{(4)}(T)&=&P_k(T).\label{V4}
\end{eqnarray}

Next we define three weak gradient operators: (1) the regular weak gradient $\nabla_{w,0}$ on volume $T$, (2) the  surface weak gradient $\nabla_{w,f}$ on face $f$, and (3) the edge weak gradient or directional derivative $\nabla_{w,e}$ on $e$. More precisely,
for $v=\{v_0,v_f,v_e, v_n\}\in V_k^{(1)}(T)$, the weak gradient on volume $T$, denoted by $\nabla_{w,0} v\in V_{k,0}^{(2)}(T)$, is defined on $T$ by
\begin{equation}\label{wg0}
(\nabla_{w,0} v, \ \varphi)_T = -( v_0, \ \nabla\cdot\varphi)_T+( v_f,\varphi\cdot\bn)_\pT
\quad
\forall \varphi\in V_{k,0}^{(2)}(T).
\end{equation}
The surface weak gradient on the face $f\in F(T)$, denoted by $\nabla_{w,f} v\in V_{k,f}^{(2)}(f)$, is defined as follows:
\begin{equation}\label{wgf}
( \nabla_{w,f}v,\ \theta\times\bn_f)_{f}=-( v_f, \nabla\times\theta\cdot\bn_f)_f+ ( v_e,\; \theta\cdot\bt_{\partial f})_{\partial f}\quad\forall\theta\in V_{k,f}^{(2)}(f),
\end{equation}
where $\bt_{\partial f}$ is chosen such that $\bt_{\partial f}$ and $\bn_f$ obey the right-hand rule. Analogously, the edge weak gradient or directional derivative on edge $e$, denoted as $\nabla_{w,e} v\in V_{k,e}^{(2)}(e)$, is defined on $e$ such that
\begin{equation}\label{wge}
( \nabla_{w,e}v, \varphi\bt_e)_{e}=-( v_e, \nabla\varphi\cdot\bt_e)_e+ ( v_n,\;\varphi \bt_e\cdot\bn_{\partial e})_{\partial e} \quad\forall\varphi\bt_e\in V_{k,e}^{(2)}(e),
\end{equation}
where $\bn_{\partial e}$ is the unit outward direction at the two end points of $e$. It is clear that $( v_n,\;\varphi \bt_e\cdot\bn_{\partial e})_{\partial e}$ is in fact the  difference of the value of $v_n\varphi$ at  two end points of $e$ with sign determined by $\bt_e$.

With the help of these weak gradient operators, we may define a composite weak gradient operator $\nabla_w: V_k^{(1)}(T)\mapsto V_k^{(2)}(T)$ as follows:
\begin{equation}\label{wg}
\nabla_wv :=\{\nabla_{w,0}v, \nabla_{w,f}v, \nabla_{w,e}v \} \in V_k^{(2)}(T)
\end{equation}
for all $v\in V_k^{(1)}(T)$.

Next, we shall introduce two weak curl operators, denoted as $\nabla_{w,0}\times$ and $\nabla_{w,f}\times$, on $T$ and $f$ respectively.
For any $\bu=\{\bu_0,\bu_f,\bu_e\}\in V_k^{(2)}(T)$, the  weak curl $\nabla_{w,0}\times\bu\in V_{k,0}^{(3)}(T)$ is defined on $T$ by
\begin{equation}\label{wc0}
  (\nabla_{w,0}\times\bu, \theta)_T = (\bu_0, \nabla\times\theta)_T+ ( \bu_f, \theta\times\bn)_{\pT}\qquad
   \forall \theta\in V_{k,0}^{(3)}(T).
\end{equation}
The surface weak curl on each face $f\in F(T)$, denoted as $\nabla_{w,f}\times\bu\in V_{k,f}^{(3)}(f)$, is defined on $f$ satisfying
\begin{equation}\label{wcf}
( \nabla_{w,f}\times\bu, \tau\bn_f)_{f}=( \bu_f, \nabla \tau\times\bn_f)_f+(\bu_e,\; \tau\bt)_{\partial f}\quad\forall\tau\bn_f\in V_{k,f}^{(3)}(f).
\end{equation}
The composite weak curl operator $\nabla_w\times: V_k^{(2)}(T)\mapsto V_k^{(3)}(T)$ is then given by setting
\begin{equation}\label{wc}
\nabla_w\times\bu=\{\nabla_{w,0}\times\bu, \nabla_{w,f}\times\bu\}.
\end{equation}

Finally, for any $\bw=\{\bw_0,\bw_f\}\in V_k^{(3)}(T)$, we define its weak divergence $\nabla_w\cdot\bw\in V_k^{(4)}(T)$ by the following equation:
\begin{equation}\label{wd}
(\nabla_{w}\cdot\bw, \tau)_T = -(\bw_0,\nabla \tau)_T+ (\bw_f,
\tau\bn)_{\partial T}\qquad \forall \tau\in V_k^{(4)}(T).
\end{equation}
It is clear that the weak divergence operator $\nabla_w\cdot$ maps $V_k^{(3)}(T)$ to $V_k^{(4)}(T)$.

\subsection{WG de Rham complex for polynomials of descending order}
For a given integer $k\ge 3$, we introduce four polynomial spaces with descending orders for the diagram in Fig. \ref{d2}: $W_k^{(1)}(T)$, $W_{k-1}^{(2)}(T)$, $W_{k-2}^{(3)}(T)$ and $W_{k-3}^{(4)}(T)$. The first polynomial space $W_k^{(1)}(T)$ is given by
\begin{equation}\label{W1}
\begin{split}
W_k^{(1)}(T)=&\{v=\{v_0,v_f,v_e, v_n\}: v_0\in P_k(T),\;v_f\in P_{k-1}(f),\; v_e\in P_{k-2}(e),\\
&v_n\in \mathbb{R},\;f\in F(T), \; e\in E(T), n\in V(T)\}.
\end{split}
\end{equation}
The second space $W_k^{(2)}(T)$ is given by
\begin{equation}\label{W2}
\begin{split}
W_{k-1}^{(2)}(T)=&\{ \bu=\{\bu_0, \bu_f,\bu_e\}:\;\bu_0\in W_{k-1,0}^{(2)}(T), \bu_f\in W_{k-2,f}^{(2)}(f),\\
&\bu_e\in W_{k-3,e}^{(2)}(e),\;f\in F(T),\; e\in E(T)\},
\end{split}
\end{equation}
where
\begin{eqnarray}
W_{k-1,0}^{(2)}(T)&=&[P_{k-1}(T)]^3,\label{W20}\\
W_{k-2,f}^{(2)}(f)&=&\{\bu_f=u_1\bt_{f,1}+u_2\bt_{f,2}:\; u_1,u_2\in P_{k-2}(f)\}, \label{W2f}\\
W_{k-3,e}^{(2)}(e)&=&\{\bu_e=u\bt_e: u\in P_{k-3}(e)\}.\label{W2e}
\end{eqnarray}
The third space $W_k^{(3)}(T)$ is defined as
\begin{eqnarray}
W_k^{(3)}(T)&=&\{ \bw=\{\bw_0, \bw_f\}: \bw_0 \in W_{k-2,0}^{(3)}(T),\; \bw_f\in W_{k-3,f}^{(3)}(f)\},\label{W3}
\end{eqnarray}
where
\begin{eqnarray}
W_{k-2,0}^{(3)}(T)&=&[P_{k-2}(T)]^3,\label{W30}\\
W_{k-3,f}^{(3)}(f)&=&\{\bw_f=w_f\bn_f:\;  w_f\in P_{k-3}(f)\}, \label{W3f}
\end{eqnarray}
with $\bn_f$ a unit normal vector to the face $f$. The fourth polynomial space $W_k^{(4)}(T)$ is defined as follows:
\begin{eqnarray}
W_{k-3}^{(4)}(T)&=&P_{k-3}(T).\label{W4}
\end{eqnarray}

Analogously,  we define three weak gradient operators $\nabla_{w,0}$, $\nabla_{w,f}$ and  $\nabla_{w,e}$ on $T$, $f$ and $e$ accordingly. More precisely,
for $v=\{v_0,v_f,v_e, v_n\}\in W_k^{(1)}(T)$, the weak gradient operator on volume $T$, denoted as $\nabla_{w,0} v\in W_{k-1,0}^{(2)}(T)$, is defined by
\begin{equation}\label{wg10}
(\nabla_{w,0} v, \ \varphi)_T = -( v_0, \ \nabla\cdot\varphi)_T+( v_f,\varphi\cdot\bn)_\pT
\quad
\forall \varphi\in W_{k-1,0}^{(2)}(T).
\end{equation}
The surface weak gradient operator on face $f\in F(T)$, denoted as $\nabla_{w,f} v\in W_{k-2,f}^{(2)}(f)$, is defined on $f$ satisfying
\begin{equation}\label{wg1f}
( \nabla_{w,f}v,\ \theta\times\bn_f)_{f}=-(v_f, \nabla\times\theta\cdot\bn_f)_f+( v_e,\; \theta\cdot\bt_{\partial f})_{\partial f}\quad\forall\theta\in W_{k-2,f}^{(2)}(f).
\end{equation}
The edge weak gradient or directional derivative on edge $e\in E(T)$, denoted as $\nabla_{w,e} v\in W_{k-3,e}^{(2)}(e)$, is defined on $e$ such that
\begin{equation}\label{wg1e}
( \nabla_{w,e}v, \varphi\bt_e)_{e}=-( v_e, \nabla\varphi\cdot\bt_e)_e+ ( v_n,\;\varphi \bt_e\cdot\bn_{\partial e})_{\partial e} \quad\forall\varphi\bt_e\in W_{k-3,e}^{(2)}(e).
\end{equation}
Collectively, we define a weak gradient mapping $\nabla_w: W_k^{(1)}(T)\mapsto W_{k-1}^{(2)}(T)$ as follows:
\begin{equation}\label{wg1}
\nabla_w v :=\{\nabla_{w,0}v, \nabla_{w,f}v, \nabla_{w,e}v \}.
\end{equation}

Next, we  define two weak curl operators $\nabla_{w,0}\times$ and $\nabla_{w,f}\times$ on $T$ and $f$ respectively.
For any $\bu=\{\bu_0,\bu_f\}\in W_{k-1}^{(2)}(T)$, the weak curl on volume $T$, denoted as $\nabla_{w,0}\times\bu\in W_{k-2,0}^{(3)}(T)$, is defined on $T$ by
\begin{equation}\label{wc10}
  (\nabla_{w,0}\times\bu, \theta)_T = (\bu_0, \nabla\times\theta)_T+( \bu_f, \theta\times\bn)_{\pT}\qquad
   \forall \theta\in W_{k-2,0}^{(3)}(T).
\end{equation}
The surface weak curl on face $f\in F(T)$, denoted as $\nabla_{w,f}\times\bu\in W_{k-3,f}^{(3)}(f)$, is defined on $f$ by
\begin{equation}\label{wc1f}
( \nabla_{w,f}\times\bu, \tau\bn_f)_{f}=( \bu_f, \nabla \tau\times\bn_f)_f+\l\bu_e,\; \tau\bt\r_{\partial f}\quad\forall\tau\bn\in W_{k-3,f}^{(3)}(f).
\end{equation}
The composite weak curl mapping $\nabla_w\times: W_{k-1}^{(2)}(T)\mapsto W_{k-2}^{(3)}(T)$ is then defined by
\begin{equation}\label{wc1}
\nabla_w\times\bu :=\{\nabla_{w,0}\times\bu, \nabla_{w,f}\times\bu\}.
\end{equation}

Finally, the weak divergence $\nabla_w\cdot\bw\in W_{k-3}^{(4)}(T)$ for any $\bw=\{\bw_0,\bw_f\}\in W_{k-2}^{(3)}$ is defined on $T$ by the following equation
\begin{equation}\label{wd1}
(\nabla_{w}\cdot\bw, \tau)_T = -(\bw_0,\nabla \tau)_T+ (\bw_f,
\tau\bn)_{\partial T}\qquad \forall \tau\in W_{k-3}^4(T).
\end{equation}

\section{Complex Sequences}

The goal of this section is to show that the WG sequences in Fig. \ref{d1} and Fig. \ref{d2} are complexes in the sense that the composition of any two consecutive operations is zero.

\begin{lemma}
For the weak gradient operator $\nabla_w$ and the weak curl operator $\nabla_w\times$ defined in (\ref{wg}) and (\ref{wc}) or in (\ref{wg1})  and (\ref{wc1}), the following identity holds true
\begin{equation}\label{d2d1}
\nabla_w\times (\nabla_wv)=0
\end{equation}
for all $v$ in $V_k^{(1)}(T)$ or $W_k^{(1)}(T)$.
\end{lemma}

\begin{proof}
By the definition of $\nabla_w\times$ in (\ref{wc}) or (\ref{wc1}), we need to show $\nabla_{w,0}\times (\nabla_wv)=0$ and $\nabla_{w,f}\times (\nabla_wv)=0$.

For any $v$ in $V_k^{(1)}(T)$ or $W_k^{(1)}(T)$, we have $\nabla_{w} v=\{\nabla_{w,0} v, \nabla_{w,f} v, \nabla_{w,e}v\}$. By letting $\bu=\nabla_wv$ in (\ref{wc0}) or (\ref{wc10}) and using (\ref{wg0}), (\ref{wgf}), (\ref{wg10})  and (\ref{wg1f}), we have for any $\theta$ in $V_{k,0}^{(3)}(T)$ or $W_{k-2,0}^{(3)}(T)$ that
\begin{equation*}
\begin{split}
&( \nabla _{w,0}\times (\nabla_w v), \theta)_T\\
=& (\nabla_{w,0}\; v, \nabla \times\theta )_T+( \nabla_{w,f} v, \theta\times\bn )_{\pT}\\
=& -(v_0, \nabla \cdot (\nabla \times \theta ))_T +( v_f, (\nabla \times \theta )\cdot \bn)_{\pT}+( \nabla_{w,f} v, \theta\times\bn)_{\pT}\\
=& ( v_f, (\nabla \times \theta) \cdot \bn)_{\pT}- ( v_f, (\nabla \times \theta) \cdot \bn )_{\pT}+\sum_{f\in F(T)}\l v_e, \theta\cdot\bt\r_{\partial f}\\
=& 0,
\end{split}
\end{equation*}
where we have used the fact that $\sum_{f\in F(T)}( v_e, \theta\cdot\bt)_{\partial f}=0$, as each edge is shared by two adjacent faces with tangential vector $\bt$ of opposite directions.

Next we show that $\nabla_{w,f}\times(\nabla_wv)=0$ on each face $f\in F(T)$. To this end, on any given face $f\in F(T)$, by letting $\bu=\nabla_wv$ in (\ref{wcf}) or (\ref{wc1f}) and using (\ref{wgf}), (\ref{wge}), (\ref{wg1f}) and (\ref{wg1e}), we have for any $\tau$ in $V_{k,f}^{(3)}(f)$ or $W_{k-3,f}^{(3)}(f)$,
\begin{eqnarray*}
( \nabla_{w,f}\times (\nabla_wv), \tau\bn_f)_{f}&=&( \nabla_{w,f}v, \nabla \tau\times\bn_f)_f+ ( \nabla_{w,e}v, \tau\bt_{\partial f})_{\partial f}\\
&=&-( v_f, \nabla\times (\nabla\tau)\cdot\bn_f)_f+( v_e,\;\nabla\tau\cdot \bt_{\partial f})_{\partial f}+( \nabla_{w,e}v, \tau\bt_{\partial f})_{\partial f}\\
&=&( v_e,\;\nabla\tau\cdot \bt_{\partial f})_{\partial f}-( v_e,\;\nabla\tau\cdot \bt_{\partial f})_{\partial f}+\sum_{e\in \partial f}( v_n, \tau \bt_{\partial f} \cdot \bn_{\partial(\partial f)})_{\partial e}\\
&=&\sum_{e\in \partial f}( v_n, \tau \bt_{\partial f} \cdot \bn_{\partial(\partial f)})_{\partial e} =0.
\end{eqnarray*}
This completes the proof of the lemma.
\end{proof}

\begin{lemma}
For the weak curl  and the weak divergence operator $\nabla_w\times$ and $\nabla_w\cdot$ define in (\ref{wc}) and (\ref{wd}) or (\ref{wc1}) and (\ref{wd1}) respectively, the
following identity holds true
\begin{equation}\label{d3d2}
\nabla_w\cdot (\nabla_w\times\bu)=0
\end{equation}
for all $\bu$ in $V_k^{(2)}(T)$ or $W_{k-1}^{(2)}(T)$.
\end{lemma}

\begin{proof} Recall that for $\bu=\{\bu_0,\bu_f,\bu_e\}$ in $V_k^{(2)}(T)$ or $W_{k-1}^{(2)}(T)$, we have $\nabla_w\times\bu=\{\nabla_{w,0}\times\bu,\nabla_{w,f}\times\bu\}$. By letting $\bw=\nabla_w\times\bu$ in (\ref{wd}) or (\ref{wd1}) and using (\ref{wc0}) and (\ref{wcf}) or (\ref{wc10}) and (\ref{wc1f}), we have for any $\tau$ in $V_k^{(4)}(T)$ or $W_{k-3}^{(4)}(T)$ that
\begin{eqnarray*}
(\nabla_{w}\cdot (\nabla_w\times\bu), \tau)_T &=& -(\nabla_{w,0}\times\bu, \nabla\tau)_T+ ( \nabla_{w,f}\times\bu,\tau\bn)_{\pT}\\
&=&-(\bu_0,\nabla\times \nabla\tau)_T-( \bu_f,\nabla\tau\times\bn)_\pT\\
&&+( \bu_f,\nabla\tau\times\bn)_\pT+\sum_{f \subset \pT} (\bu_e,\;\tau\bt)_{\partial f}\\
&=&0
\end{eqnarray*}
which implies (\ref{d3d2}) and thus completes the proof of the lemma.
\end{proof}

\section{Commutative Properties for the WG de Rham Complex 1}
Denote by $Q_0^{(1)}$,  $Q_f^{(1)}$ and $Q_e^{(1)}$ the $L^2$ projection operators onto $P_k(T)$, $P_k(f)$, and $P_k(e)$ respectively. For any $v\in H^1(T)\cap C(T)$, define $Q_h^{(1)}v$ by
\begin{equation}\label{Q1}
Q_h^{(1)}v=\{Q_0^{(1)}v,Q_f^{(1)}v|_f, Q_e^{(1)}v|_e,v|_n\}\in V_k^{(1)}(T).
\end{equation}
Next, let $Q_0^{(2)}$, $Q_f^{(2)}$ and $Q_e^{(2)}$ be the $L^2$ projection operators onto $V_{k,0}^{(2)}(T)$, $V_{k,f}^{(2)}(f)$, and $V_{k,e}^{(2)}(e)$ respectively. For $\bu\in H(curl;T)\cap [C(T)]^3$, define $Q_h^{(2)}\bu$ by
\begin{equation}\label{Q2}
Q_h^{(2)}\bu=\{Q_0^{(2)}\bu,Q_f^{(2)}(\bn_f\times(\bu|_f\times\bn_f)),Q_e^{(2)}(\bu|_e\cdot\bt_e) \bt_e\}\in V_k^{(2)}(T).
\end{equation}
Analogously, with $Q_0^{(3)}$ and $Q_f^{(3)}$ being the $L^2$ projection operators onto $V_{k,0}^{(3)}(T)$ and $V_{k,f}^{(3)}(f)$ we define
\begin{equation}\label{Q3}
Q_h^{(3)}\bw=\{Q_0^{(3)}\bw,Q_f^{(3)}(\bw|_f\cdot\bn_f)\bn_f\}\in V_k^{(3)}(T)
\end{equation}
for any $\bw\in H(div;T)\cap [C(T)]^3$. Denote by $Q_h^{(4)}$ the $L^2$ projection operator from $L^2(T)$ to $V_k^{(4)}(T)$.

\begin{lemma}
For the projection operators $Q_h^{(1)}$ and $Q_h^{(2)}$ defined in (\ref{Q1}) and (\ref{Q2}) respectively and the weak gradient $\nabla_w$  defined in (\ref{wg}), we have
\begin{equation}\label{cm1}
\nabla_w (Q_h^{(1)}v)=Q_h^{(2)}\nabla v\qquad \forall v\in H^1(T)\cap C^1(T).
\end{equation}
\end{lemma}

\begin{proof}
By (\ref{wg}), one has $\nabla_wQ_h^{(1)}v=\{\nabla_{w,0}Q_h^{(1)}v, \nabla_{w,f}Q_h^{(1)}v, \nabla_{w,e} Q_h^{(1)}v\}$. From (\ref{Q2}), we have
$Q_h^{(2)}\nabla v=\{Q_0^{(2)}\nabla v,Q_f^{(2)}(\bn_f\times(\nabla v|_f\times\bn_f)),Q_e^{(2)}(\nabla v|_e\cdot\bt_e) \bt_e\}$.
Thus it suffices to prove
\begin{eqnarray}\label{g1.one}
\nabla_{w,0}Q_h^{(1)}v&=&Q_0^{(2)}\nabla v,\quad \mbox{ in $T$},\\
\nabla_{w,f}Q_h^{(1)}v&=&Q_f^{(2)}(\bn_f\times(\nabla v \times\bn_f)),\quad \mbox{ on $f\in F(T)$},\label{g1.two}\\
\nabla_{w,e}Q_h^{(1)}v&=&Q_e^{(2)}(\nabla v\cdot\bt_e) \bt_e,\quad \mbox {on $e\in E(T)$}. \label{g1.three}
\end{eqnarray}
To prove \eqref{g1.one}, we have from (\ref{wg0}) that for $v\in H^1(T)\cap C^1(T)$ and any $\varphi\in V_{k,0}^{(2)}(T)$,
\begin{eqnarray*}
(\nabla_{w,0} Q_h^{(1)}v, \ \varphi)_T &=& -( Q_0^{(1)}v, \ \nabla\cdot\varphi)_T+( Q_f^{(1)}v,\varphi\cdot\bn)_\pT\\
&=&-(v, \ \nabla\cdot\varphi)_T+( v,\varphi\cdot\bn)_\pT\\
&=&(\nabla v, \ \varphi)_T=(Q_0^{(2)}\nabla v, \ \varphi)_T,
\end{eqnarray*}
which leads to $\nabla_{w,0}Q_h^{(1)}v=Q_0^{(2)}\nabla v$ in $T$, and thus proves \eqref{g1.one}.

Next, from (\ref{wgf}) we have for any $\theta\in V_{k,f}^{(2)}(f)$
\begin{eqnarray*}
( \nabla_{w,f}Q_h^{(1)}v,\ \theta\times\bn_f)_{f}&=&-( Q_f^{(1)}v, \nabla\times\theta\cdot\bn_f)_f+ ( Q_e^{(1)}v,\; \theta\cdot\bt)_{\partial f}\\
&=&-( v, \nabla\times\theta\cdot\bn_f)_f+ ( v,\; \theta\cdot\bt)_{\partial f}\\
&=&( \nabla v,\ \theta\times\bn)_{f}=( \bn_f\times(\nabla v\times\bn_f),\ \theta\times\bn_f)_{f}\\
&=& ( Q_f^{(2)}(\bn_f\times(\nabla v\times\bn_f)),\ \theta\times\bn)_{f},
\end{eqnarray*}
which verifies the identity \eqref{g1.two}.

To derive \eqref{g1.three}, from (\ref{wge}) we have for $v\in H^1(T)\cap C^1(T)$ and any $\varphi\in V_{k,e}^{(2)}(e)$ that
\begin{eqnarray*}
( \nabla_{w,e}Q_h^{(1)}v, \varphi\bt_e)_{e}&=&-( Q_e^{(1)}v, \nabla\varphi\cdot\bt_e)_e+ ( v,\;\varphi \bt_e\cdot\bn_{\partial e})_{\partial e}\\
&=&-( v, \nabla\varphi\cdot\bt_e)_e+ ( v,\;\varphi \bt_e\cdot\bn_{\partial e})_{\partial e}\\
&=&( \nabla v, \varphi\bt_e)_{e}=( Q_e^{(2)}(\nabla v\cdot\bt_e)\bt_e, \varphi\bt_e)_{e}
\end{eqnarray*}
which verifies \eqref{g1.three}. This completes the proof of the lemma.
\end{proof}

\begin{lemma}
For the projections $Q_h^{(2)}$ and $Q_h^{(3)}$ defined in (\ref{Q2}) and (\ref{Q3}) respectively and the weak curl $\nabla_w\times$  defined in (\ref{wc}), we have
\begin{equation}\label{cm2}
\nabla_w\times (Q_h^{(2)}\bu)=Q_h^{(3)}\nabla\times\bu\qquad \forall \bu\in H(curl, T)\cap [C^1(T)]^3.
\end{equation}
\end{lemma}

\begin{proof}
By (\ref{wc}), one has $\nabla_w\times Q_h^{(2)}\bu=\{\nabla_{w,0}\times Q_h^{(2)}\bu, \nabla_{w,f}\times Q_h^{(2)}\bu\}$. By (\ref{Q3}),  $Q_h^{(3)}\nabla\times\bu=\{Q_0^{(3)}\nabla\times\bu, Q_f^{(3)}(\nabla\times\bu\cdot\bn_f)\bn_f\}$. Thus, if suffices to prove
\begin{eqnarray}\label{c1:01}
\nabla_{w,0}\times Q_h^{(2)}\bu&=&Q_0^{(3)}\nabla\times\bu,\\
\nabla_{w,f}\times Q_h^{(2)}\bu&=&Q_f^{(3)}(\nabla\times\bu\cdot\bn_f)\bn_f. \label{c1:02}
\end{eqnarray}
It follows from (\ref{wc0}) that for $\bu\in H(curl, T)\cap[C^1(T)]^3$ and any $\theta\in V_{k,0}^{(3)}(T)$,
\begin{eqnarray*}
 (\nabla_{w,0}\times Q_h^{(2)}\bu, \theta)_T &=& (Q_0^{(2)}\bu, \nabla\times\theta)_T + ( Q_f^{(2)}(\bn\times(\bu\times\bn)), \theta\times\bn)_{\pT}\\
&=& (\bu, \nabla\times\theta)_T + ( \bn\times(\bu\times\bn), \theta\times\bn)_{\pT}\\
&=& (\bu, \nabla\times\theta)_T + ( \bu, \theta\times\bn)_{\pT}\\
 &=&(\nabla\times \bu, \theta)_T=(Q_0^{(3)}\nabla\times \bu, \theta)_T
\end{eqnarray*}
which verifies \eqref{c1:01}.
Next, from (\ref{wcf}) we have for any $\tau\bn_f\in V_{k,f}^{(3)}(f)$,
\begin{eqnarray*}
(\nabla_{w,f}\times Q_h^{(2)}\bu, \tau\bn_f)_{f}&=&( Q_f^{(2)}(\bn_f\times(\bu\times\bn_f)), \nabla \tau\times\bn_f)_f+( Q_e^{(2)}(\bu\cdot\bt)\bt,\; \tau\bt)_{\partial f}\\
&=&( \bn_f\times(\bu\times\bn_f), \nabla \tau\times\bn_f)_f+( (\bu\cdot\bt)\bt,\; \tau\bt)_{\partial f}\\
&=&( \bu, \nabla \tau\times\bn_f)_f+( \bu,\; \tau\bt)_{\partial f}\\
&=&( \nabla\times \bu \cdot\bn_f, \tau)_{f}=( Q_f^{(3)}(\nabla\times \bu\cdot\bn_f)\bn_f, \tau\bn_f)_{f},
\end{eqnarray*}
which leads to \eqref{c1:02}. This completes the proof of the lemma.
\end{proof}

\begin{lemma}
For the operator $Q_h^{(3)}$ defined in (\ref{Q3}), the $L^2$ projection operator $Q_h^{(4)}$ from $L^2(T)$ to $V_k^{(4)}(T)$, and the weak divergence operator $\nabla_w\cdot$  defined in (\ref{wd}), the following identity holds true
\begin{equation}\label{cm3}
\nabla_w\cdot (Q_h^{(3)}\bw)=Q_h^{(4)}\nabla\cdot\bw\qquad \forall \bw\in H(div; T)\cap [C(T)]^3.
\end{equation}
\end{lemma}

\begin{proof}
It follows from (\ref{wd}) that for $\bw\in H(div,T)\cap [C(T)]^3$ and any $\tau\in V_k^{(4)}(T)$,
\begin{eqnarray*}
(\nabla_{w}\cdot Q_h^{(3)}\bw, \tau)_T &=& -(Q_0^{(3)}\bw,\nabla \tau)_T+ ( Q_f^{(3)}(\bw\cdot\bn)\bn, \tau\bn)_{\pT}\\
&=& -(\bw,\nabla \tau)_T+ (\bw, \tau\bn)_{\pT}\\
&=&(\nabla\cdot \bw, \tau)_T=(Q_h^{(4)}\nabla\cdot \bw, \tau)_T,
\end{eqnarray*}
which completes the proof of the lemma.
\end{proof}

\section{Commutative Properties for the WG de Rham Complex 2}

In this section, we show that the diagram in Fig. \ref{d2} commutes with properly defined operators $R_h^{(1)}$, $R_h^{(2)}$, $R_h^{(3)}$ and  $R_h^{(4)}$. To this end, denote by $R_0^{(1)}$, $R_f^{(1)}$ and $R_e^{(1)}$ the $L^2$ projection operators onto $P_k(T)$, $P_{k-1}(f)$, and $P_{k-2}(e)$ respectively.  For any $v\in H^1(T)\cap C^1(T)$, we define $R_h^{(1)}v$ as follows
\begin{equation}\label{R1}
R_h^{(1)}v=\{R_0^{(1)}v,R_f^{(1)}v|_f,R_e^{(1)}v|_e,v|_n\}\in W_k^{(1)}(T).
\end{equation}
Next, denote by $R_0^{(2)}$, $R_f^{(2)}$ and $R_e^{(2)}$ the $L^2$ projection operators onto $W_{k-1,0}^{(2)}(T)$, $W_{k-2,f}^{(2)}(f)$, and $W_{k-3,e}^{(2)}(e)$ respectively. For $\bu\in H(curl;T)\cap [C(T)]^3$, we define $R_h^{(2)}\bu$ as follows
\begin{equation}\label{R2}
R_h^{(2)}\bu=\{R_0^{(2)}\bu,R_f^{(2)}(\bn_f\times(\bu|_f\times\bn_f)),R_e^{(2)}(\bu|_e\cdot\bt_e)\bt_e\}\in W_{k-1}^{(2)}(T).
\end{equation}
Analogously, with $R_0^{(3)}$ and $R_f^{(3)}$ being the $L^2$ projection operators onto $W_{k-2,0}^{(3)}(T)$ and $W_{k-3,f}^{(3)}(f)$, we may define $R_h^{(3)}\bw$ by
\begin{equation}\label{R3}
R_h^{(3)}\bw=\{R_0^{(3)}\bw,R_f^{(3)}(\bw\cdot\bn_f)\bn_f\}\in W_{k-2}^{(3)}(T)
\end{equation}
for all $\bw\in H(div;T)\cap[C(T)]^3$. Our fourth operator $R_h^{(4)}$ is given as the $L^2$ projection operator from $L^2(T)$ to $W_{k-3}^{(4)}(T)$.

\begin{lemma}
For the linear operators $R_h^{(1)}$ and $R_h^{(2)}$ defined in (\ref{R1}) and (\ref{R2}) respectively and the weak gradient $\nabla_w$  defined in (\ref{wg1}), the following identity holds true
\begin{equation}\label{cm11}
\nabla_w (R_h^{(1)}v)=R_h^{(2)}\nabla v
\end{equation}
for all $v\in H^1(T)\cap C^1(T)$,
\end{lemma}

\begin{proof}
It follows from (\ref{wg1}) that
$\nabla_wR_h^{(1)}v=\{\nabla_{w,0}R_h^{(1)}v, \nabla_{w,f}R_h^{(1)}v, \nabla_{w,e}R_h^{(1)}v\}$. By (\ref{R2}),  we have
$$
R_h^{(2)}\nabla v=\{R_0^{(2)}\nabla v, R_f^{(2)}(\bn_f\times(\nabla v\times\bn_f)), R_e^{(2)}(\nabla v\cdot\bt_e)\bt_e\}.
$$
Thus we need to prove
\begin{eqnarray}\label{g11.one}
\nabla_{w,0}R_h^{(1)}v&=&R_0^{(2)}\nabla v,\qquad \mbox{ in } T,\\
\nabla_{w,f}R_h^{(1)}v&=&R_f^{(2)}(\bn_f\times(\nabla v\times\bn_f)),\qquad\mbox{ on } f\in F(T),\label{g11.two}\\
\nabla_{w,e}R_h^{(1)}v&=&R_e^{(2)}(\nabla v\cdot\bt_e)\bt_e,\qquad \mbox{ on } e\in E(T).\label{g11.three}
\end{eqnarray}
From (\ref{wg10}), we have for any $\varphi\in W_{k-1,0}^{(2)}(T)$
\begin{eqnarray*}
(\nabla_{w,0} R_h^{(1)}v, \ \varphi)_T &=& -( R_0^{(1)}v, \ \nabla\cdot\varphi)_T+(R_f^{(1)}v,\varphi\cdot\bn)_\pT\\
&=&-(v, \ \nabla\cdot\varphi)_T+(v,\varphi\cdot\bn)_\pT\\
&=&(\nabla v, \ \varphi)_T=(R_0^{(2)}\nabla v, \ \varphi)_T,
\end{eqnarray*}
which proves \eqref{g11.one}.

Next, from (\ref{wg1f}) we have for any $\theta\in W_{k-2,f}^{(2)}(f)$
\begin{eqnarray*}
(\nabla_{w,f}R_h^{(1)}v,\ \theta\times\bn_f)_{f}&=&-( R_f^{(1)}v, \nabla\times\theta\cdot\bn_f)_f+( R_e^{(1)}v,\; \theta\cdot\bt)_{\partial f}\\
&=&-(v, \nabla\times\theta\cdot\bn_f)_f+( v,\; \theta\cdot\bt)_{\partial f}\\
&=&( \nabla v,\ \theta\times\bn_f)_{f}\\
&=&( \bn_f\times(\nabla v\times\bn_f),\ \theta\times\bn_f)_{f}\\
&=&( R_f^{(2)}(\bn_f\times(\nabla v\times\bn_f)),\ \theta\times\bn_f)_{f},
\end{eqnarray*}
which verifies \eqref{g11.two}.

Finally, from (\ref{wg1e}), we have for any $\varphi\in W_{k-3,e}^{(2)}(e)$
\begin{eqnarray*}
(\nabla_{w,e}R_h^{(1)}v, \varphi\bt_e)_{e}&=&-( R_e^{(1)}v, \nabla\varphi\cdot\bt_e)_e+ ( v,\;\varphi\bt_e\cdot\bn_{\partial e})_{\partial e}\\
&=&-( v, \nabla\varphi\cdot\bt_e)_e+ ( v,\;\varphi\bt_e\cdot\bn_{\partial e})_{\partial e}\\
&=&( \nabla v\cdot\bt_e, \varphi)_{e}=( R_e^{(2)}(\nabla v\cdot\bt_e)\bt_e, \varphi\bt_e)_{e},
\end{eqnarray*}
which proves \eqref{g11.three}. This completes the proof of the lemma.
\end{proof}

\begin{lemma}
For the linear operators $R_h^{(2)}$ and $R_h^{(3)}$ defined in (\ref{R2}) and (\ref{R3}) and the weak curl $\nabla_w\times$  defined in (\ref{wc1}), the following identity holds true:
\begin{equation}\label{cm12}
\nabla_w\times (R_h^{(2)}\bu)=R_h^{(3)}\nabla\times\bu
\end{equation}
for all $\bu\in H(curl; T)\cap [C(T)]^3$.
\end{lemma}

\begin{proof}
By (\ref{wc1}), one has $\nabla_w\times R_h^{(2)}\bu=\{\nabla_{w,0}\times R_h^{(2)}\bu, \nabla_{w,f}\times R_h^{(2)}\bu\}$. By (\ref{R3}), we have $R_h^{(3)}\nabla\times\bu=\{R_0^{(3)}\nabla\times\bu, R_f^{(3)}(\nabla\times\bu\cdot\bn_f)\bn_f)\}$. Thus it suffices to prove
\begin{equation}\label{c11}
\nabla_{w,0}\times R_h^{(2)}\bu=R_0^{(3)}\nabla\times\bu,\;\;\nabla_{w,f}\times R_h^{(2)}\bu=R_f^{(3)}(\nabla\times\bu\cdot\bn_f)\bn_f.
\end{equation}
First, it follows from (\ref{wc10}) that for $\bu\in H(curl; T)\cap[C(T)]^3$ and any $\theta\in W_{k-2,0}^{(3)}(T)$,
\begin{eqnarray*}
 (\nabla_{w,0}\times R_h^{(2)}\bu, \theta)_T &=& (R_0^{(2)}\bu, \nabla\times\theta)_T +( R_f^{(2)}(\bn\times(\bu\times\bn)), \theta\times\bn)_{\pT}\\
&=& (\bu, \nabla\times\theta)_T +(\bn\times(\bu\times\bn), \theta\times\bn)_{\pT}\\
&=& (\bu, \nabla\times\theta)_T + ( \bu, \theta\times\bn)_{\pT}\\
 &=&(\nabla\times \bu, \theta)_T=(R_0^{(3)}\nabla\times \bu, \theta)_T,
\end{eqnarray*}
which implies that $\nabla_{w,0}\times R_h^{(2)}\bu=R_0^{(3)}\nabla\times\bu$.

Next, from (\ref{wc1f}) we have for any $\tau\bn\in W_{k-3,f}^{(3)}(f)$
\begin{eqnarray*}
(\nabla_{w,f}\times R_h^{(2)}\bu, \tau\bn_f)_{f}&=&( R_f^{(2)}(\bn_f\times(\bu\times\bn_f)), \nabla \tau\times\bn_f)_f+( R_e^{(2)}(\bu\cdot\bt)\bt,\; \tau\bt)_{\partial f}\\
&=&( \bn_f\times(\bu\times\bn_f), \nabla \tau\times\bn_f)_f+((\bu\cdot\bt)\bt,\; \tau\bt)_{\partial f}\\
&=&( \bu, \nabla \tau\times\bn_f)_f+( \bu,\; \tau\bt)_{\partial f}\\
&=&( \nabla\times \bu\cdot\bn_f, \tau)_{f}=( R_f^{(3)}(\nabla\times \bu\cdot\bn_f)\bn_f, \tau\bn_f)_{f},
\end{eqnarray*}
which implies $\nabla_{w,f}\times R_h^{(2)}\bu=R_f^{(3)}(\nabla\times\bu\cdot\bn_f)\bn_f$ on face $f\in F(T)$. This completes the proof of the lemma.
\end{proof}

\begin{lemma}
For the linear operator $R_h^{(3)}$ defined in (\ref{R3}),  the $L^2$ projection $R_h^{(4)}$  from $L^2(T)$ to $W_{k-3}^{(4)}(T)$, and the weak divergence $\nabla_w\cdot$  defined in (\ref{wd1}), the following identity holds true
\begin{equation}\label{cm13}
\nabla_w\cdot (R_h^{(3)}\bw)=R_h^{(4)}\nabla\cdot\bw
\end{equation}
for all $\bw\in H(div,T)\cap[C(T)]^3$.
\end{lemma}

\begin{proof}
It follows from (\ref{wd1}) that for $\bw\in H(div,T)\cap[C(T)]^3$ and any $\tau\in W_{k-3}^{(4)}(T)$, we have
\begin{eqnarray*}
(\nabla_{w}\cdot R_h^{(3)}\bw, \tau)_T &=& -(R_0^{(3)}\bw,\nabla \tau)_T+ ( R_f^{(3)}(\bw\cdot\bn)\bn, \tau\bn)_{\pT}\\
&=& -(\bw,\nabla \tau)_T+ ( \bw\cdot\bn, \tau)_{\pT}\\
&=&(\nabla\cdot \bw, \tau)_T=(R_h^{(4)}\nabla\cdot \bw, \tau)_T,
\end{eqnarray*}
which proves the lemma.
\end{proof}

\section{Exactness for de Rham Complex 1}
We show that the weak Galerkin de Rham complex 1 is exact for $k=0$ on tetrahedra and  hexahedra.

\begin{theorem}
Let $T$ be a tetrahedron.  The following de Rham complex is exact for $k=0$
 \begin{equation} \label{c0}
\begin{split} \hbox{
\begin{tikzpicture}[scale=1.4]
 \node (AA) at (-1,1) {${\mathbb R}^1$};
 \node (A) at (0,1) {$C^\infty(T)$};
 \node (B) at (1.7,1) {$[{C}^\infty(T)]^3$};
 \node (C) at (3.4,1) {$[{C}^\infty(T)]^3$};
 \node (D) at (5.1,1) {$C^\infty(T)$};
 \node (I) at (6.0,1) {$0$};
 \node (EE) at (-1,0)   {${\mathbb R}^4$};
 \node (E) at (0,0)   {$V_{k}^{(1)}(T)$};
 \node (F) at (1.7,0) {$V_{k}^{(2)}(T)$};
 \node (G) at (3.4,0) {$V_{k}^{(3)}(T)$};
 \node (H) at (5.1,0) {$V_{k}^{(4)}(T)$};
 \node (II) at (6.0,0) {$0$};
 \draw[->] (AA) to node [above]{$I$}(A);
 \draw[->] (A) to node [above]{$\nabla$}(B);
 \draw[->] (B) to node [above]{$\nabla\times$}(C);
 \draw[->] (C) to node [above]{$\nabla\cdot$}(D);
 \draw[->] (D) to node [above]{$N$}(I);
 \draw[->] (EE) to node [above]{$I_w$}(E);
 \draw[->] (E) to node [above]{$\nabla_w$}(F);
 \draw[->] (F) to node [above]{$\nabla_w\times$}(G);
 \draw[->] (G) to node [above]{$\nabla_w\cdot$}(H);
 \draw[->] (A) to node [left]{$Q_h^{(1)}$}(E);
 \draw[->] (B) to node [left]{$Q_h^{(2)}$}(F);
 \draw[->] (C) to node [left]{$Q_h^{(3)}$}(G);
 \draw[->] (D) to node [left]{$Q_h^{(4)}$}(H);
 \draw[->] (H) to node [above]{$N$}(II);
 \end{tikzpicture}
}
\end{split}
\end{equation}
Here $I_w$ is the inclusion map that assigns constant value to $v_0$, $v_f$, $v_e$, and $v_n$; $N$ stands for the null operator.
\end{theorem}

\begin{proof} A necessary condition for exactness is zero difference of dimensions; i.e.,
\begin{equation}\label{EQ:100}
\begin{split}
    &\quad \ \sum_{i=0}^4 (-1)^{i} \dim V_k^{(i)}(T) \\
& = \quad \ \ 4 \\
    &\quad \  - \frac{(k+1)(k+2)(k+3)}6 - \quad \ 4\frac{(k+1)(k+2)}2 - 6(k+1)-4\\
     &\quad \ + 3\frac{(k+1)(k+2)(k+3)}6 +2\cdot 4\frac{(k+1)(k+2)}2 + 6(k+1) \\
     &\quad \ - 3\frac{(k+1)(k+2)(k+3)}6  - \quad \ 4\frac{(k+1)(k+2)}2\\
    &\quad \  + \ \frac{(k+1)(k+2)(k+3)}6 \\
&=\quad \ \ 0,
\end{split}
\end{equation}
where we have set $V_k^{(0)}(T) = {\mathbb R}^4$.

We claim that the kernel of the operator $\nabla_{w}$ has dimension $4$. To this end,  let $v\in  V_0^{(1)}(T)\in\kernel(\nabla_w)$; i.e.
$$
0=\nabla_w v =\{\nabla_{w,0} v, \nabla_{w,f}v, \nabla_{w,e}v\}.
$$
From the definition \eqref{wge} for $\nabla_{w,e}v$, we have $v_n=\alpha_4$ with a constant $\alpha_4$ at all the vertices of $T$. On each face (triangle) $f\in F(T)$, one may construct a linear function $P_{1,f} v$ so that $P_{1,f} v= v_e$ at the center of each edge $e\in \partial f$. From the defition \eqref{wgf} we have
$$
 \nabla_{w,f}v = \nabla_f  P_{1,f} v \qquad \forall f\in F(T),
$$
where $\nabla_f$ is the surface gradient operator on the face $f$. It follows from $\nabla_{w,f}v=0$ that $\nabla_f  P_{1,f} v=0$ so that $P_{1,f} v=\alpha_3$ for a constant $\alpha_3$. This shows that $v_e=\alpha_3$ on all edges. Analogously, we may obtain $v_f=\alpha_2$ on all faces for a constant $\alpha_2$. Finally, $v_0=\alpha_1$ is already a constant that does not enter into the calculation of the weak derivatives for the case of $k=0$. This shows that the dimension of $\kernel(\nabla_w)$ is $4$ so that
$$
\range(I_w) = \kernel(\nabla_w).
$$
Since the dimension of $V_0^{(1)}(T)=15$, thus the dimension of $\range(\nabla_w)=15-4=11$.

Next, we claim that
\begin{equation}\label{r(curl)}
\dim (\range(\nabla_w\times))=6.
\end{equation}

Since  $\dim (V_0^{(4)})=1$ and $\dim (V_0^{(3)})=7$, we have $\dim (\range(\nabla_w \cdot ))=1$ and $\dim(\kernel (\nabla_w\cdot))=6$. It follows that
$$
\range(\nabla_w\times) = \kernel (\nabla_w\cdot)
$$
provided that \eqref{r(curl)} holds true.

From $\dim (V_0^{(2)})= \dim(\range(\nabla_w\times)) + \dim(\kernel(\nabla_w\times))$ and \eqref{r(curl)}, we have
$$
\dim(\kernel(\nabla_w\times)) = \dim (V_0^{(2)}) - 6 = 17-6=11=\dim(\range(\nabla_w)).
$$
Hence, we have from \eqref{d2d1} that
$$
\kernel(\nabla_w\times) = \range(\nabla_w).
$$

It remains to prove \eqref{r(curl)}. Note that by (\ref{d3d2}), we have $\dim (\range(\nabla_w\times))\le \dim(\kernel (\nabla_w\cdot))=6$. Consider the following orthogonal decomposition of $V_0^{(3)}(T)$:
\[
V_0^{(3)}(T)= \range(\nabla_w\times)\oplus W.
\]
It is clear that \eqref{r(curl)} is equivalent to $\dim (W)=1$. For any $\bw=\{\bw_0,w_f\bn_f\}\in W$, we have
\[
(\bw, \nabla_w\times\bv)_T=0\quad\forall \bv\in V_0^{(2)},
\]
which implies
\begin{eqnarray}
(\bw_0, \nabla_{w,0}\times\bv)_T&=&0,\label{ww0}\\
(w_f\bn_f, \nabla_{w,f}\times\bv)_f&=&0\quad \mbox{ on each face } f.\label{wwf}
\end{eqnarray}
Using the definition of $\nabla_{w,0}$ in (\ref{wc0}) and (\ref{ww0}), we have
\[
\sum_f(\bw_0,\bn_f\times\bv_f)_f=0\quad\forall \bv=\{\bv_0,\bv_f\}\in V_{0,0}^{(2)}\times V_{0,f}^{(2)},
\]
which implies $\bw_0\times\bn_f=0$ on each face $f$ so that $\bw_0=0$. Next, it follows from (\ref{wwf}) that
\begin{eqnarray*}
\sum_f (w_f\bn_f,\nabla_{w,f}\times \bv)_f=\sum_e([w_f], v_e)_e=0,
\end{eqnarray*}
which implies the jump $[w_f]_e=0$ on all edges. Thus, $w_f$ assumes a constant value at all faces $f\in F(T)$. This shows that the function $\bw\in W$ have the form $\bw=\{0,  c\bn_f\}$ with $c=const$ so that $\dim (W)=1$. This verifies the claim (\ref{r(curl)}).
\end{proof}

\begin{theorem}
For $k=0$,  the following de Rham complex is exact on cubic element $T$:
 \begin{equation} \label{c0c} 
\begin{split}\hbox{
\begin{tikzpicture}[scale=1.4]
 \node (AA) at (-1,1) {${\mathbb R}^1$};
 \node (A) at (0,1) {$C^\infty(T)$};
 \node (B) at (1.7,1) {$[{C}^\infty(T)]^3$};
 \node (C) at (3.4,1) {$[{C}^\infty(T)]^3$};
 \node (D) at (5.1,1) {$C^\infty(T)$};
 \node (I) at (6.0,1) {$0$};
 \node (EE) at (-1,0)   {${\mathbb R}^8$};
 \node (E) at (0,0)   {$V_{k}^{(1)}(T)$};
 \node (F) at (1.7,0) {$V_{k}^{(2)}(T)$};
 \node (G) at (3.4,0) {$V_{k}^{(3)}(T)$};
 \node (H) at (5.1,0) {$V_{k}^{(4)}(T)$};
 \node (II) at (6.0,0) {$0$};
 \draw[->] (AA) to node [above]{$I$}(A);
 \draw[->] (A) to node [above]{$\nabla$}(B);
 \draw[->] (B) to node [above]{$\nabla\times$}(C);
 \draw[->] (C) to node [above]{$\nabla\cdot$}(D);
 \draw[->] (D) to node [above]{$N$}(I);
 \draw[->] (EE) to node [above]{$I_w$}(E);
 \draw[->] (E) to node [above]{$\nabla_w$}(F);
 \draw[->] (F) to node [above]{$\nabla_w\times$}(G);
 \draw[->] (G) to node [above]{$\nabla_w\cdot$}(H);
 \draw[->] (A) to node [left]{$Q_h^{(1)}$}(E);
 \draw[->] (B) to node [left]{$Q_h^{(2)}$}(F);
 \draw[->] (C) to node [left]{$Q_h^{(3)}$}(G);
 \draw[->] (D) to node [left]{$Q_h^{(4)}$}(H);
 \draw[->] (H) to node [above]{$N$}(II);
 \end{tikzpicture}
}
\end{split}
\end{equation}
\end{theorem}

\begin{proof} Again, the necessary condition of zero difference of dimensions for exactness holds true, as it can be easily seen that
\begin{align*}
    &\quad \ \sum_{i=0}^4 (-1)^{i-1} \dim V_k^{(i)}(T) \\
& = -8 \\
    & \quad \ +  \frac{(k+1)(k+2)(k+3)}6 + \quad \ 6\frac{(k+1)(k+2)}2 + 12(k+1)+8\\
     &\quad \ - 3\frac{(k+1)(k+2)(k+3)}6 -2\cdot 6\frac{(k+1)(k+2)}2 - 12(k+1) \\
     &\quad \ + 3\frac{(k+1)(k+2)(k+3)}6  + \quad \ 6\frac{(k+1)(k+2)}2\\
    &\quad \  - \ \frac{(k+1)(k+2)(k+3)}6 =0,
\end{align*}
where we have set $V_k^{(0)}(T) = {\mathbb R}^8$.

Let us show that the kernel of the operator $\nabla_{w}$ has dimension $8$. In fact,  for any $v\in  V_0^{(1)}(T)\in\kernel(\nabla_w)$, we have
$$
0=\nabla_w v =\{\nabla_{w,0} v, \nabla_{w,f}v, \nabla_{w,e}v\}.
$$
From the definition \eqref{wge} for $\nabla_{w,e}v$, we see that $v_n=\alpha_8$ with a constant $\alpha_8$ at all eight vertices of $T$. On each face (rectangle) $f\in F(T)$, the condition of $\nabla_{w,f}v=0$ implies $v_e$ has the same value on any parallel edges so that $v_e$ has a total of 3 independent unknowns. Analogously, the condition of $\nabla_{w,0} v = 0$ implies that $v_f$ has the same value on any parallel faces so that $v_f$ has a total of 3 independent unknowns. With the one free unknown for $v_0$, we have a total of $8$ independent unknowns for functions in $\kernel (\nabla_w)$ so that
\begin{equation}\label{EQ:4-26:001}
\dim(\kernel (\nabla_w)) = 8,
\end{equation}
which leads to
$$
\range(I_w) = \kernel(\nabla_w).
$$
Since the dimension of $V_0^{(1)}(T)=27$, thus the dimension of $\range(\nabla_w)=27-8=19$.

Observe that the proof of \eqref{r(curl)} can be adopted without any modification to yield the following result:
\begin{equation}\label{r(curl)-new}
\dim (\range(\nabla_w\times))=8.
\end{equation}

Since  $\dim (V_0^{(4)})=1$ and $\dim (V_0^{(3)})=9$, we then have $\dim (\range(\nabla_w \cdot ))=1$ and $\dim(\kernel (\nabla_w\cdot))=8$. It follows from \eqref{r(curl)-new} that
$$
\range(\nabla_w\times) = \kernel (\nabla_w\cdot).
$$

Next, from $\dim (V_0^{(2)})= \dim(\range(\nabla_w\times)) + \dim(\kernel(\nabla_w\times))$ and \eqref{r(curl)-new}, we have
$$
\dim(\kernel(\nabla_w\times)) = \dim (V_0^{(2)}) - 8 = 27-8=19=\dim(\range(\nabla_w)).
$$
Hence, we have from \eqref{d2d1} that
$$
\kernel(\nabla_w\times) = \range(\nabla_w).
$$
\end{proof}

\end{document}